\documentclass{article}
\usepackage{amssymb}
\usepackage{amsmath}
\usepackage{amsthm}
\usepackage{graphicx}
\usepackage{delarray}
\usepackage{tikz-cd}
\usepackage{hyperref}
\usepackage{cleveref}
\providecommand{\coloneqq}{\mathrel{:=}}
\providecommand{\eqqcolon}{\mathrel{=:}}
\newcommand{\lb}{\begin{array}[t][{@{}c@{}}]\displaystyle}
\newcommand{\rb}{\end{array}}
\newcommand{\op}[1]{\operatorname{#1}}
\newcommand{\textem}[1]{\textbf{#1}}
\newtheorem{definition}{Definition}
\newtheorem{theorem}[definition]{Theorem}
\newtheorem{proposition}[definition]{Proposition}
\newtheorem{lemma}[definition]{Lemma}

\crefname{definition}{Definition}{Definitions}
\crefname{theorem}{Theorem}{Theorems}
\crefname{proposition}{Proposition}{Propositions}
\crefname{lemma}{Lemma}{Lemmata}
\crefname{corollary}{Corollary}{Corollaries}

\newenvironment{cd}[1]{\begin{tikzcd}[ampersand replacement=\&, #1]}{\end{tikzcd}}
\tikzset{commutative diagrams/arrow style=math font}
\newcommand{\ssp}{\;}
\newcommand{\cl}{\colon}

\newcommand{\ocdots}{\mathord{\cdots}}
\newcommand{\dsm}{\delta}
\newcommand{\esm}{\varepsilon}
\newcommand{\psm}{\pi}
\newcommand{\dbl}{\delta _{\mathrm{BL}}}
\newcommand{\ebl}{\varepsilon _{\mathrm{BL}}}
\newcommand{\pbl}{\pi _{\mathrm{BL}}}
\newcommand{\treet}{\hat{T}}
\newcommand{\tjoin}{\mathbin{\wedge}}
\newcommand{\ajoin}{\mathbin{\hat{\otimes}}}

\newcommand{\trivtree}{\bullet}
\newcommand{\empseq}{\emptyset}
\newcommand{\untree}{\mathclose{\downarrow}}
\newcommand{\rmL}{\mathrm{L}}
\newcommand{\rmR}{\mathrm{R}}
\newcommand{\bbK}{\mathbb{K}}
\newcommand{\lP}{\langle}
\newcommand{\mP}{\mid}
\newcommand{\rP}{\rangle}
\newcommand{\lM}{\{}
\newcommand{\rM}{\}}
\newcommand{\leaf}{\op{leaf}}
\newcommand{\pdeg}{\op{deg}}
\newcommand{\pcan}{\op{can}}
\newcommand{\prV}{\alpha _V}
\newcommand{\prK}{\alpha _\bbK}

\newcommand{\rmid}{\mathrm{id}}
\newcommand{\coeff}[2]{C (#1, #2)}
\newcommand{\rep}[1]{{\op{Rep} (#1)}}

\newcommand{\nor}[1]{{\op{Nor} (#1)}}
\newcommand{\blrep}[1]{{\op{Rep} _{\mathrm{BL}} (#1)}}
\tikzset{
  point/.style={draw, circle, fill, inner sep=0pt, minimum size=2pt},
}
\title{Explicit~construction of~cofree~precoalgebras and~coalgebras}
\author{Yuki Goto}
\date{}
\begin{document}
\maketitle

\begin{abstract}
  In this paper, we will explicitly construct cofree coalgebras in the category of modules on a commutative ring.
  Our approach does not impose any condition to the coefficient ring, which means, for example, it does not need to be a field.
  The main technique is to generate homomorphisms from certain type of infinite trees so that those constitute cofree precoalgebras.
\end{abstract}

\section{Introduction}

Given a module $ V $, a coalgebra $ \overline{V} $ with a couniversal homomorphism $ \pi \cl \overline{V} \to V $ is called a cofree coalgebra.
Although this notion first arose in algebraic topology, it is also known to be worthwhile in categorical logic, in particular Girard's linear logic \cite{gllx}.
In fact, based on a Lafont's idea, Bierman showed that in a category with cofree coalgebras, we can interpret the exponential operator of linear logic as a functor which sends an object to the cofree coalgebra over it \cite{llcm,bill}.

Until now, several explicit constructions of cofree coalgebras have been introduced.
For example, Sweedler constructed them by using the notion of finite dual, which was generalised by Anquela and Cort\'es to those over arbitrary algebraic varieties \cite{hax,accx}.
Block and Leroux presented them as the vector space of all representative linear maps \cite{bgdc}.
Hazewinkel used the notion of tensor power series, and Murfet used local cohomology and residues \cite{hccm,mscc}.
However, all of these constructions require some conditions on the coefficient ring, mainly being a field.
In this paper, we will present a new approach to explicitly construct cofree coalgebras, which does not impose any condition to the coefficient ring.

In our approach, we will first consider precoalgebras (namely coalgebras not necessarily coassociative nor counital), and construct cofree precoalgebras.
We will then construct cofree coalgebras from them.
There are two reasons why we first construct cofree precoalgebras, instead of making cofree coalgebras directly.
The first is that it is fairly easier to manipulate precoalgebras than coalgebras, since we do not need to consider coassociativity or counitality.
The second is that a precoalgebra structure is needed to construct a categorical model of linear logic.
There is a general method, called Chu construction, to create a model of linear logic from a monoidal closed category, but here we need a cofree precoalgebra structure rather than a cofree coalgebra structure \cite{ccsa,bacm}.

Our main technique is to consider some sort of infinite trees and generate homomorphisms from them.
This idea was once adopted by Fox, but his argument was pointed out to contain an error by Hazewinkel \cite{fccc,hccm}.
This paper is also meant to fix this error.

The outline of this paper is as follows.
In Section 2, we will recall the definition of precoalgebras and coalgebras and their cofreeness.
We will also recall the Barr's result which provides a method of creating cofree coalgebras from cofree precoalgebras.
Section 3 is the main part of this paper, which concerns an explicit construction of cofree precoalgebras and cofree coalgebras.
In Section 4, we will briefly compare the Block and Leroux's construction with ours by explicitly presenting an isomorphism between them.

\section{Precoalgebras and coalgebras}

In this section, we will recall the definition of precoalgebras and coalgebras, and also that of their cofreeness.
Throughout this paper, we fix a commutative ring (not necessarily a field) as a coefficient ring of modules, and let $ \bbK $ denote it.
We suppose all modules, homomorphisms between them and the tensor products of them are all over $ \bbK $.

\begin{definition}
  A module $ C $ is called a \textem{precoalgebra} when it is equipped with two homomorphisms $ \dsm \cl C \to C \otimes C $ and $ \esm \cl C \to \bbK $, which satisfy no other additional conditions.
\end{definition}

\begin{definition}
  A module $ C $ is called a \textem{coalgebra} when it is equipped with two homomorphisms $ \dsm \cl C \to C \otimes C $ and $ \esm \cl C \to \bbK $, and these homomorphisms make the following diagrams commute:
  \begin{gather*}
    \begin{cd}{}
      C \ar[r, "\dsm"] \ar[d, "\dsm"'] \& C \otimes C \ar[d, "\dsm \otimes \rmid"]\\
      C \otimes C \ar[r, "\rmid \otimes \dsm"'] \& C \otimes C \otimes C
    \end{cd} \qquad
    \begin{cd}{}
      \& C \ar[d, "\dsm"] \ar[dr, ""] \ar[dl, ""'] \& \\
      \bbK \otimes C \& C \otimes C \ar[l, "\esm \otimes \rmid"] \ar[r, "\rmid \otimes \esm"'] \& C \otimes \bbK
    \end{cd}
  \end{gather*}
  Here, $ C \to \bbK \otimes C $ and $ C \to C \otimes \bbK $ are the canonical isomorphisms.
\end{definition}

Clearly every coalgebra is a precoalgebra.
We will sometimes say that a precoalgebra is \textem{admissible} when it is also a coalgebra (that is, when the two diagrams above both commute).

A morphism between them is defined as follows:

\begin{definition}
  For two precoalgebras $ (C, \dsm, \esm) $ and $ (D, \dsm', \esm') $, a homomorphism $ \varphi \cl D \to C $ is called a \textem{precoalgebra morphism} when the following diagrams both commute:
  \begin{gather*}
    \begin{cd}{}
      D \ar[r, "\varphi"] \ar[d, "\dsm'"'] \& C \ar[d, "\dsm"] \\
      D \otimes D \ar[r, "\varphi \otimes \varphi"'] \& C \otimes C
    \end{cd} \qquad
    \begin{cd}{column sep=small}
      D \ar[rr, "\varphi"] \ar[rd, "\smash{\esm'} \vphantom{\esm}"'] \& \& C \ar[ld, "\esm"] \\
      \& \bbK \&
    \end{cd}
  \end{gather*}
  A coalgebra morphism is defined in the exactly same way.
\end{definition}

Finally, the cofreeness of them is defined as follows:

\begin{definition}
  For a module $ V $, a precoalgebra $ (\overline{V}, \dsm, \esm) $ is said to be \textem{cofree} over $ V $ when the following two conditions both hold:
  \begin{itemize}
    \item $ \overline{V} $ is equipped with a homomorphism $ \psm \cl \overline{V} \to V $;
    \item for any precoalgebra $ (D, \dsm', \esm') $ with a homomorphism $ \varphi \cl D \to V $, there exists a unique precoalgebra morphism $ \tilde{\varphi} \cl D \to \overline{V} $ which makes the following diagram commute:
    \begin{gather*}
      \begin{cd}{}
        D \ar[d, "\tilde{\varphi}"'] \ar[dr, "\varphi"] \& \\
        \overline{V} \ar[r, "\psm"'] \& V 
      \end{cd}
    \end{gather*}
  \end{itemize}
  The cofreeness of a coalgebra is also similarly defined.
\end{definition}

Barr showed that, once cofree precoalgebras are constructed, cofree coalgebras can be given as the largest admissible subprecoalgebras of them.
He proved this fact in any cocomplete and well-powered categories, but here we restate it in the category of modules to fit our settings.

\begin{theorem}[Barr, Theorem 4.3 \cite{bacm}] \label{barr}
  Assume that there exists the cofree precoalgebra $ (\overline{V}, \dsm, \esm) $ over any module $ V $.
  Then the largest admissible subprecoalgebra of $ (\overline{V}, \dsm, \esm) $ is the cofree coalgebra over every module $ V $.
\end{theorem}

In the next section, we will first present an explicit construction of cofree precoalgebras.
We will furthermore construct the largest admissible subprecoalgebras of them, which are proved to be cofree coalgebras by this theorem.

\section{Explicit construction}

\subsection{Cofree precoalgebras}

In this section, we will introduce two types of trees: one is called a grading tree, and the other is called a generating tree.
The former is used as a parametre of a special type of modules considered throughout this section (see \cref{defgrt,deftga}), and the latter is used to generate a homomorphism (see \cref{defgnt,defgenhom}).

A grading tree is defined as follows:

\begin{definition} \label{defgrt}
  By a \textem{grading tree}, we mean a finite full binary tree, that is, a rooted tree such that the number of nodes are finite and every node has exactly zero or two child nodes.
\end{definition}

\newcommand{\extreeleft}{
  \tikz[baseline={([yshift=-.8ex]current bounding box.center)}]{
    \path (0, 0) node [point] (O) {};
    \path (-0.2, -0.3) node [point] (L) {};
    \path (0.2, -0.3) node [point] (R) {};
    \draw (O) -- (L); \draw (O) -- (R);
  }
}
\newcommand{\extree}{
  \tikz[baseline={([yshift=-.8ex]current bounding box.center)}]{
    \path (0, 0) node [point] (O) {};
    \path (-0.2, -0.3) node [point] (L) {};
    \path (0.2, -0.3) node [point] (R) {};
    \path (-0.35, -0.6) node [point] (LL) {};
    \path (-0.05, -0.6) node [point] (LR) {};
    \draw (O) -- (L); \draw (O) -- (R);
    \draw (L) -- (LL); \draw (L) -- (LR);
  }
}

For two grading trees $ t $ and $ u $, we write $ t \tjoin u $ for the grading tree obtained by adding a new root which has two edges to the roots of $ t $ and $ u $.
Moreover, we write $ \trivtree $ for the trivial grading tree, which has no node other than the root.
For example, when $ t = \extreeleft $ and $ u = \trivtree $, we have $ t \tjoin u = \extree $.

Using grading trees, we generalise the notion of graded algebra.
Usual graded algebras are parametrised by natural numbers, while our new type of algebras are parametrised by grading trees.

\begin{definition} \label{deftga}
  Suppose that a module $ A ^t $ is given for each grading tree $ t $.
  Their direct sum
  \begin{gather*}
    A \coloneqq \bigoplus _t A ^t
  \end{gather*}
  is called a \textem{tree-graded algebra} when it is equipped with a homomorphism $ \ajoin \cl A ^t \otimes A ^u \to A ^{t \tjoin u} $ for any pair of grading trees $ (t, u) $.
\end{definition}

\begin{definition}
  For two tree-graded algebras $ A $ and $ B $, a homomorphism $ f \cl A \to B $ is called \textem{tree-graded} when we have $ f (A ^t) \subseteq B ^t $ for any grading tree $ t $.
\end{definition}

Starting from any module $ V $, we can construct a tree-graded algebra.
The construction is similar to usual tensor algebra.

\begin{definition}
  Take a module $ V $.
  For each grading tree $ t $, define
  \begin{gather*}
    \treet ^t V \coloneqq V ^{\otimes {\leaf (t)}} = \underbrace{V \otimes \cdots \otimes V} _{\leaf (t) \text{ times}},
  \end{gather*}
  where $ \leaf (t) $ denotes the number of leaves in $ t $.
  Moreover, for two grading trees $ t $ and $ u $, let
  \begin{gather*}
    \begin{array}{r@{}r@{}c@{}l}
      \ajoin \cl {} & \treet ^t V \otimes \treet ^u V & {} \longrightarrow {} & \treet ^{t \tjoin u} V \\
      & (v _1 \otimes \cdots \otimes v _n) \otimes (w _1 \otimes \cdots \otimes w _m) & {} \longmapsto {} & v _1 \otimes \cdots \otimes v _n \otimes w _1 \otimes \cdots \otimes w _m.
    \end{array}
  \end{gather*}
  With this homomorphism, the direct sum
  \begin{gather*}
    \treet V \coloneqq \bigoplus _t \treet ^t V
  \end{gather*}
  is a tree-graded algebra, which we call a \textem{tree-tensor algebra} of $ V $.
\end{definition}

Note that, unlike that of a usual tensor algebra, the operation $ \ajoin $ defined above is not (and never) associative, since $ \treet ^{(t \tjoin u) \tjoin v} V \ne \treet ^{t \tjoin (u \tjoin v)} V \subseteq \treet V $ for grading trees $ t $, $ u $, $ v $.

Now consider in particular the case $ V = \bbK $.
Since we have
\begin{gather*}
  \treet \bbK = \bigoplus _t \treet ^t \bbK = \bigoplus _t \bbK ^{\otimes {\leaf (t)}} \cong \bigoplus _t \bbK,
\end{gather*}
the tree-tensor algebra $ \treet \bbK $ is the free module generated by all grading trees.
This means that, to define a tree-graded homomorphism $ f \cl \treet \bbK \to B $ where $ B $ is some tree-graded algebra, we only need to determine $ f (t) \in B ^t $ for each grading tree $ t $.

Given a module $ V $, our cofree precoalgebra over $ V $ is presented as a submodule of $ \op{Hom} (\treet \bbK, \treet (V \oplus \bbK)) $, whose elements are tree-graded ones generated by another type of trees, which we call generating trees:

\begin{definition} \label{defgnt}
  For a module $ V $, a \textem{generating tree} for $ V $ is an infinite tree $ \sigma $ which satisfies the following two conditions:
  \begin{itemize}
    \item each node of $ \sigma $ has an even number of two or more children; if the node has exactly $ 2 n \ssp (n \geq 1) $ children, they are named $ \rmL (1) $, $ \ocdots $, $ \rmL (n) $, $ \rmR (1) $, $ \ocdots $, $ \rmR (n) $ in order from left to right;
    \item each node of $ \sigma $ is labelled by an element of $ V \oplus \bbK $.
  \end{itemize}
\end{definition}

Each node of a generating tree $ \sigma $ can be uniquely specified by a finite sequence of symbols of the form $ \rmL (i) $ or $ \rmR (i) $, by reading the name of the child successively from the root; the root node is specified by the empty sequence $ \empseq $ (see Figure \ref{gentree}).
We call such a finite sequence a \textem{position sequence}, and often identify it with the node specified by it.
For a position sequence $ r $, we write $ \sigma \lP r \rP $ for the label attached to the node specified by $ r $, and $ \sigma \untree _r $ for the subtree under that node (so the root of $ \sigma \untree _r $ is the node of $ \sigma $ specified by $ r $).

\begin{figure}[tbp]
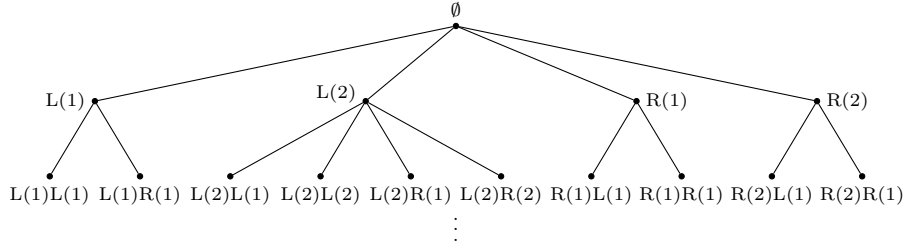

  \centering
  \tikz{
    \path (5.4, 0) node [point] (O) {} node [above] {\scriptsize$ \empseq $};
    \path (0.6, -1) node [point] (L1) {} node [left] {\scriptsize$ \rmL (1) $};
    \path (4.2, -1) node [point] (L2) {} ++ (0, 0.1) node [left] {\scriptsize$ \rmL (2) $};
    \path (7.8, -1) node [point] (R1) {} node [right] {\scriptsize$ \rmR (1) $};
    \path (10.2, -1) node [point] (R2) {} node [right] {\scriptsize$ \rmR (2) $};
    \path (0, -2) node [point] (L1L1) {} node [below] {\scriptsize$ \rmL (1) \rmL (1)$};
    \path (1.2, -2) node [point] (L1R1) {} node [below] {\scriptsize$ \rmL (1) \rmR (1) $};
    \path (2.4, -2) node [point] (L2L1) {} node [below] {\scriptsize$ \rmL (2) \rmL (1)$};
    \path (3.6, -2) node [point] (L2L2) {} node [below] {\scriptsize$ \rmL (2) \rmL (2) $};
    \path (4.8, -2) node [point] (L2R1) {} node [below] {\scriptsize$ \rmL (2) \rmR (1)$};
    \path (6, -2) node [point] (L2R2) {} node [below] {\scriptsize$ \rmL (2) \rmR (2) $};
    \path (7.2, -2) node [point] (R1L1) {} node [below] {\scriptsize$ \rmR (1) \rmL (1)$};
    \path (8.4, -2) node [point] (R1R1) {} node [below] {\scriptsize$ \rmR (1) \rmR (1) $};
    \path (9.6, -2) node [point] (R2L1) {} node [below] {\scriptsize$ \rmR (2) \rmL (1)$};
    \path (10.8, -2) node [point] (R2R1) {} node [below] {\scriptsize$ \rmR (2) \rmR (1) $};
    \draw (O) -- (L1); \draw (O) -- (L2); \draw (O) -- (R1); \draw (O) -- (R2);
    \draw (L1) -- (L1L1); \draw (L1) -- (L1R1);
    \draw (L2) -- (L2L1); \draw (L2) -- (L2L2); \draw (L2) -- (L2R1); \draw (L2) -- (L2R2);
    \draw (R1) -- (R1L1); \draw (R1) -- (R1R1);
    \draw (R2) -- (R2L1); \draw (R2) -- (R2R1);
    \path (5.4, -2.6) node {\scriptsize$ \vdots $}
  }
  \caption{Generating tree and position sequences} \label{gentree}
\end{figure}

Here we see how a generating tree generates a tree-graded homomorphism.

\begin{definition} \label{defgenhom}
  Given a generating tree $ \sigma $ for a module $ V $, define a tree-graded homomorphism $ f \cl \treet \bbK \to \treet (V \oplus \bbK) $, which is said to be \textem{generated} by $ \sigma $, as follows.
  First set $ f (\trivtree) \coloneqq \sigma \lP \empseq \rP $.
  For a grading tree $ t $ other than $ \trivtree $, we define $ f (t) $ in the following way.

  Just like the naming of children in a generating tree, call the left child $ \rmL $ and the right child $ \rmR $.
  Then every node of $ t $ is uniquely specified by a finite sequence of $ \rmL $ or $ \rmR $.
  Let $ p _1, \ocdots, p _n $ be all the leaves of $ t $ in order from left to right.
  For each inner node $ q $ in $ t $, prepare a variable symbol $ \iota \lP q \rP $ which ranges over positive integers.
  
  For $ 1 \leq k \leq n $, write $ p _k \eqqcolon \xi _{k, 1} \cdots \xi _{k, m _k} $ where each $ \xi _{k, l} $ is either $ \rmL $ or $ \rmR $.
  Set
  \begin{gather*}
    r _k \coloneqq \xi _{k, 1} (\iota \lP \empseq \rP) \, \xi _{k, 2} (\iota \lP \xi _{k, 1} \rP) \cdots \xi _{k, m _k} (\iota \lP \xi _{k, 1} \cdots \xi _{k, m _k - 1} \rP),
  \end{gather*}
  which is a finite sequence of symbols of the form $ \rmL (i) $ or $ \rmR (i) $.
  Now define
  \begin{gather*}
    f (t) \coloneqq \sum _\iota (\sigma \lP r _1 \rP \otimes \sigma \lP r _2 \rP \otimes \cdots \otimes \sigma \lP r _n \rP),
  \end{gather*}
  where the sum ranges over $ (r _1, r _2, \ocdots, r _n) $ such that every $ r _k \ssp (1 \leq k \leq n) $ specifies some nodes in  $ \sigma $ when $ \iota \lP q \rP $'s run through positive integers.
  This sum is obviously finite, and thus determines an element of $ \treet ^t (V \oplus \bbK) $.
\end{definition}

For example, consider the case $ t = \extree $.
Since $ t $ has three leaves $ p _1 \coloneqq \mathrm{LL} $, $ p _2 \coloneqq \mathrm{LR} $, $ p _3 \coloneqq \rmR $, we have
\begin{align*}
  r _1 & = \rmL (\iota \lP \empseq \rP) \rmL (\iota \lP \rmL \rP) \\
  r _2 & = \rmL (\iota \lP \empseq \rP) \rmR (\iota \lP \rmL \rP) \\
  r _3 & = \rmR (\iota \lP \empseq \rP).
\end{align*}
Thus the value $ f (t) $ is defined as
\begin{align*}
  f (t) & = \sum _\iota (\sigma \lP r _1 \rP \otimes \sigma \lP r _2 \rP  \otimes \sigma \lP r _3 \rP) \\
  & = \sum _\iota (\sigma \lP \rmL (\iota \lP \empseq \rP) \rmL (\iota \lP \rmL \rP) \rP \otimes \sigma \lP \rmL (\iota \lP \empseq \rP) \rmR (\iota \lP \rmL \rP) \rP \otimes \sigma \lP \rmR (\iota \lP \empseq \rP) \rP).
\end{align*}
Rewriting $ i \coloneqq \iota \lP \empseq \rP $ and $ j \coloneqq \iota \lP \rmL \rP $ for readability, this can be also equivalently presented as
\begin{gather*}
  f (t) = \sum _{i, j} (\sigma \lP \rmL (i) \rmL (j) \rP \otimes \sigma \lP \rmL (i) \rmR (j) \rP \otimes \sigma \lP \rmR (i) \rP).
\end{gather*}
If the generating tree is that given by Figure \ref{gentree}, this sum is expanded as
\begin{align*}
  f (t) = {} & \sigma \lP \rmL (1) \rmL (1) \rP \otimes \sigma \lP \rmL (1) \rmR (1) \rP \otimes \sigma \lP \rmR (1) \rP \\
  & {} + \sigma \lP \rmL (2) \rmL (1) \rP \otimes \sigma \lP \rmL (2) \rmR (1) \rP \otimes \sigma \lP \rmR (2) \rP \\
  & {} + \sigma \lP \rmL (2) \rmL (2) \rP \otimes \sigma \lP \rmL (2) \rmR (2) \rP \otimes \sigma \lP \rmR (2) \rP.
\end{align*}

Here note that the range through which $ \iota \lP q \rP $ runs may depend on $ \iota \lP q' \rP $ for shorter $ q' $.
In the example above, $ \iota \lP \rmL \rP $ (which is also written as $ j $) can only be 1 when $ \iota \lP \empseq \rP $ (also written as $ i $) is equal to 1, while $ \iota \lP \rmL \rP $ can be 1 and 2 when $ \iota \lP \empseq \rP $ is 2.

\begin{definition}
  For a module $ V $, a homomorphism $ f \cl \treet \bbK \to \treet (V \oplus \bbK) $ is said to be \textem{representative} when $ f $ is generated by some generating tree for $ V $\footnotemark.
  Let $ \rep{V} \subseteq \op{Hom} (\treet \bbK, \treet (V \oplus \bbK)) $ denote the set of all representative homomorphisms.
\end{definition}
\footnotetext{The term ``representative'' is borrowed by Block--Leroux, but its meaning here is not the same as the original.}

Through the rest of this subsection, we will prove that $ \rep{V} $ is the cofree precoalgebra over every module $ V $.
First we will show that $ \rep{V} $ is a module.

\begin{lemma} \label{submod}
  For a module $ V $, the set $ \rep{V} $ is a submodule of $ \op{Hom} (\treet \bbK, \treet (V \oplus \bbK)) $.
\end{lemma}

\begin{proof}
  Take two generating trees $ \sigma $ and $ \tau $.
  We define a new tree $ \sigma + \tau $ in the following way. 
  First set $ (\sigma + \tau) \lP \empseq \rP \coloneqq \sigma \lP \empseq \rP + \tau \lP \empseq \rP $.
  For a nonempty position sequence $ r $ and $ s $ of $ \sigma $ and $ \tau $ respectively, write $ r \eqqcolon \xi (i) r' $ and $ s \eqqcolon \eta (j) s' $ where $ \xi $ is either $ \rmL $ or $ \rmR $ and so is $ \eta $.
  Set 
  \begin{align*}
    (\sigma + \tau) \lP \xi (i) r' \rP & \coloneqq \sigma \lP \xi (i) r' \rP \\
    (\sigma + \tau) \lP \eta (j + n) s' \rP & \coloneqq \tau \lP \eta (j) s' \rP,
  \end{align*}
  where $ n $ is half the number of the children of the root of $ \sigma $.
  In other words, $ \sigma + \tau $ is the tree obtained by the following procedure:
  \begin{itemize}
    \item remove the root of $ \sigma $ to obtain $ \sigma \untree _{\rmL (1)}, \ocdots, \sigma \untree _{\rmL (n)} $, $ \sigma \untree _{\rmR (1)}, \ocdots, \sigma \untree _{\rmR (n)} $;
    \item remove the root of $ \tau $ to obtain $ \tau \untree _{\rmL (1)}, \ocdots, \tau \untree _{\rmL (m)} $, $ \tau \untree _{\rmR (1)}, \ocdots, \tau \untree _{\rmR (m)} $;
    \item create a new root which is labelled $ \sigma \lP \empseq \rP + \tau \lP \empseq \rP $;
    \item create edges from the new root to $ \sigma \untree _{\rmL (1)}, \ocdots, \sigma \untree _{\rmL (n)} $, $ \tau \untree _{\rmL (1)}, \ocdots, \tau \untree _{\rmL (m)} $, $ \sigma \untree _{\rmR (1)}, \ocdots, \sigma \untree _{\rmR (n)} $, $ \tau \untree _{\rmR (1)}, \ocdots, \tau \untree _{\rmR (m)} $ in order from left to right.
  \end{itemize}

  Next take a generating tree $ \sigma $ and a scalar $ \lambda $.
  To define a new tree $ \lambda \sigma $, set for each position sequence $ r $ of $ \sigma $,
  \begin{gather*}
    (\lambda \sigma) \lP r \rP \coloneqq
    \begin{cases}
      \lambda \, \sigma \lP r \rP & (\text{$ r $ is of the form $ \rmL (i _1) \cdots \rmL (i _a) $, $ a \geq 0 $}) \\
      \sigma \lP r \rP & (\text{otherwise}).
    \end{cases}
  \end{gather*}

  If $ \sigma $ and $ \tau $ generate $ f $ and $ g $ respectively, then $ \sigma + \tau $ generates $ f + g $.
  Similarly, if $ \sigma $ generates $ f $, then $ \lambda \sigma $ generates $ \lambda f $.
  This proves the lemma.
\end{proof}

Now recall the following well-known fact:

\begin{proposition} \label{teninj}
  For modules $ C $ and $ D $, the homomorphism
  \begin{gather*}
    \begin{array}{r@{}r@{}c@{}l}
      \varPsi \cl {} & \op{Hom} (C, D) ^{\otimes n} & {} \longrightarrow {} & \op{Hom} (C ^{\otimes n}, D ^{\otimes n}) \\
      & g _1 \otimes \cdots \otimes g _n & {} \longmapsto {} & [ c _1 \otimes \cdots \otimes c _n \longmapsto g _1 (c _1) \otimes \cdots \otimes g _n (c _n) ]
    \end{array}
  \end{gather*}
  is injective.
\end{proposition}

To define a precoalgebra structure on $ \rep{V} $, we need the following three lemmata:

\begin{lemma} \label{inj}
  For a module $ V $, the homomorphism defined by
  \begin{gather*}
    \begin{array}{r@{}r@{}c@{}l}
      \varPhi \cl {} & \op{Hom} (\treet \bbK, \treet (V \oplus \bbK)) ^{\otimes 2} & {} \longrightarrow {} & \op{Hom} (\treet \bbK, \treet (V \oplus \bbK)) \\
      & g \otimes h & {} \longmapsto {} & \lb
      \begin{array}{@{\,}r@{}c@{}l@{\,}}
        \trivtree & {} \longmapsto {} & 0 \\
        t \tjoin u & {} \longmapsto {} & g (t) \ajoin h (u)
      \end{array}
      \rb
    \end{array}
  \end{gather*}
  is injective.
\end{lemma}

\begin{proof}
  By \cref{teninj}, the morphism
  \begin{gather*}
    \begin{array}{r@{}r@{}c@{}l}
      \varPsi _1 \cl {} & \op{Hom} (\treet \bbK, \treet (V \oplus \bbK)) ^{\otimes 2} & {} \longrightarrow {} & \op{Hom} (\treet \bbK \otimes \treet \bbK, \treet (V \oplus \bbK) \otimes \treet (V \oplus \bbK)) \\
      & g \otimes h & {} \longmapsto {} & [ t \otimes u \longmapsto g (t) \otimes h (u) ]
    \end{array}
  \end{gather*}
  is injective.
  In addition,
  \begin{gather*}
    \begin{array}{r@{}r@{}c@{}l}
      \varPsi _2 \cl {} & \op{Hom} (\treet \bbK \otimes \treet \bbK, \treet (V \oplus \bbK) \otimes \treet (V \oplus \bbK)) & {} \longrightarrow {} & \op{Hom} (\treet \bbK, \treet (V \oplus \bbK) \otimes \treet (V \oplus \bbK)) \\
      & p & {} \longmapsto {} & \lb
      \begin{array}{@{\,}r@{}c@{}l@{\,}}
        \trivtree & {} \longmapsto {} & 0 \\
        t \tjoin u & {} \longmapsto {} & p (t \otimes u)
      \end{array}
      \rb
    \end{array}
  \end{gather*}
  is also clearly injective.
  Now consider
  \begin{gather*}
    \begin{array}{r@{}r@{}c@{}l}
      \psi \cl {} & \treet (V \oplus \bbK) \otimes \treet (V \oplus \bbK) & {} \longrightarrow {} & \treet (V \oplus \bbK) \\
      & a \otimes b & {} \longmapsto {} & a \ajoin b,
    \end{array}
  \end{gather*}
  which induces
  \begin{gather*}
    \begin{array}{r@{}r@{}c@{}l}
      \varPsi _3 \cl {} & \op{Hom} (\treet \bbK, \treet (V \oplus \bbK) \otimes \treet (V \oplus \bbK)) & {} \longrightarrow {} & \op{Hom} (\treet \bbK, \treet (V \oplus \bbK)) \\
      & q & {} \longmapsto {} & \psi \circ q.
    \end{array}
  \end{gather*}
  Since $ \psi $ is clearly injective, so is $ \varPsi _3 $.
  Thus the composite $ \varPhi = \varPsi _3 \circ \varPsi _2 \circ \varPsi _1 $ is shown to be injective.
\end{proof}

\begin{lemma} \label{split}
  For a module $ V $ and a representative homomorphism $ f \cl \treet \bbK \to \treet (V \oplus \bbK) $, there exist two finite families $ (g _i) _{i \in I} $ and $ (h _i) _{i \in I} $ of representative homomorphisms such that, for any grading trees $ t $ and $ u $, we have
  \begin{gather*}
    f (t \tjoin u) = \sum _{i \in I} (g _i (t) \ajoin h _i (u)).
  \end{gather*}
\end{lemma}

\begin{proof}
  Suppose that $ f $ is generated by $ \sigma $.
  Let $ \rmL (1), \ocdots, \rmL (n), \rmR (1), \ocdots, \rmR (n) $ be all the children of the root of $ \sigma $.
  For each $ 1 \leq i \leq n $, let $ g _i $ be the homomorphism generated by $ \sigma \untree _{\rmL (i)} $, and $ h _i $ that generated by $ \sigma \untree _{\rmR (i)} $.
  Then $ (g _i) _{1 \leq i \leq n} $ and $ (h _i) _{1 \leq i \leq n} $ satisfy the desired property.
\end{proof}

\begin{lemma} \label{splituniq}
  For a module $ V $ and a representative homomorphism $ f \cl \treet \bbK \to \treet (V \oplus \bbK) $, suppose that we have finite families $ (g _i) _{i \in I}, (h _i) _{i \in I}, (g' _j) _{j \in J}, (h' _j) _{j \in J} $ of representative homomorphisms such that, for any grading trees $ t $ and $ u $, we have
  \begin{gather*}
    f (t \tjoin u) = \sum _{i \in I} (g _i (t) \ajoin h _i (u)) = \sum _{j \in J} (g' _j (t) \ajoin h' _j (u)).
  \end{gather*}
  Then we have
  \begin{gather*}
    \sum _{i \in I} (g _i \otimes h _i) = \sum _{j \in J} (g' _j \otimes h' _j)
  \end{gather*}
  as elements of $ \rep{V} \otimes \rep{V} $.
\end{lemma}

\begin{proof}
  Set
  \begin{align*}
    p & \coloneqq \sum _{i \in I} (g _i \otimes h _i) \\
    p' & \coloneqq \sum _{j \in J} (g' _j \otimes h' _j),
  \end{align*}
  and let $ \varPhi $ denote the homomorphism defined in \cref{inj}.
  For any grading trees $ t $ and $ u $, we have
  \begin{align*}
    \varPhi (p) (\trivtree) & = 0 \\
    \varPhi (p) (t \tjoin u) & = \sum _{i \in I} (g _i (t) \ajoin h _i (u)) = f (t \ajoin u),
  \end{align*}
  and $ p' $ also satisfies the exactly same formulae.
  Thus we have $ \varPhi (p) = \varPhi (p') $.
  Since $ \varPhi $ is injective, we have $ p = p' $, which proves the lemma.
\end{proof}

For a representative homomorphism $ f \cl \treet \bbK \to \treet (V \oplus \bbK) $, take $ (g _i) _{i \in I} $ and $ (h _i) _{i \in I} $ as in \cref{split} and define
\begin{gather*}
  \dsm (f) \coloneqq \sum _{i \in I} (g _i \otimes h _i).
\end{gather*}
By \cref{splituniq}, this is independent of the choice of $ (g _i) _i $ and $ (h _i) _i $ and depends only on $ f $.
This concludes that the definition above yields a homomorphism $ \dsm \cl \rep{V} \to \rep{V} \otimes \rep{V} $.
In addition, letting
\begin{align*}
  \esm (f) & \coloneqq (\text{the $ \bbK $-component of $ f (\trivtree) $}) \\
  \psm (f) & \coloneqq (\text{the $ V $-component of $ f (\trivtree) $})
\end{align*}
yields $ \esm \cl \rep{V} \to \bbK $ and $ \psm \cl \rep{V} \to V $.
We will show that these data give the cofree precoalgebra:

\begin{theorem} \label{cofmod}
  For a module $ V $, the precoalgebra $ (\rep{V}, \dsm, \esm, \psm) $ is cofree over $ V $.
\end{theorem}

Before the proof, we introduce a new notation.
For a precoalgebra $ (D, \dsm', \esm') $ and an element $ d \in D $, since $ \dsm' (d) $ is an element of $ D \otimes D $, it can be written as a finite sum of the tensor products of two elements of $ D $.
Write
\begin{gather*}
  \dsm' (d) \eqqcolon \sum _i (d _{\rmL (i)} \otimes d _{\rmR (i)})
\end{gather*}
for that sum, where $ i $ runs through some range of positive integers.
Note that such representation of $ \dsm' (d) $ is not unique, so we fix one such representation for each $ d $.
Moreover, for each $ i $, similarly write
\begin{align*}
  \dsm' (d _{\rmL (i)}) & \eqqcolon \sum _j (d _{\rmL (i) \rmL (j)} \otimes d _{\rmL (i) \rmR (j)}) \\
  \dsm' (d _{\rmR (i)}) & \eqqcolon \sum _j (d _{\rmR (i) \rmL (j)} \otimes d _{\rmR (i) \rmR (j)}),
\end{align*}
and proceed recursively to define $ d _r $ where $ r $ is a position sequence.

Using this notation, for another precoalgebra $ (C, \dsm, \esm) $ and a homomorphism $ \varphi \cl D \to C $, the commutativity of the diagram
\begin{gather*}
  \begin{cd}{}
    D \ar[r, "\varphi"] \ar[d, "\dsm'"'] \& C \ar[d, "\dsm"] \\
    D \otimes D \ar[r, "\varphi \otimes \varphi"'] \& C \otimes C
  \end{cd} 
\end{gather*}
is rephrased as
\begin{gather*}
  \sum _i (\varphi (d) _{\rmL (i)} \otimes \varphi (d) _{\rmR (i)}) = \sum _i (\varphi (d _{\rmL (i)}) \otimes \varphi (d _{\rmR (i)})).
\end{gather*}
Hence in particular, if $ \varphi $ is a precoalgebra morphism, then $ \varphi $ satisfies this formula.

Now we return to the proof of \cref{cofmod}.

\begin{proof}
  Take a precoalgebra $ (D, \dsm', \esm') $ and a homomorphism $ \varphi \cl D \to V $.
  Suppose that there exists a precoalgebra morphism $ \tilde{\varphi} \cl D \to \rep{V} $ such that
  \begin{gather*}
    \begin{cd}{}
      D \ar[d, "\tilde{\varphi}"'] \ar[dr, "\varphi"] \& \\
      \rep{V} \ar[r, "\psm"'] \& V 
    \end{cd}
  \end{gather*}
  commutes.
  We will first prove the uniqueness of such $ \tilde{\varphi} $  and then prove its existence.

  For $ d \in D $, by the definition of $ \esm $ and $ \psm $, we have
  \begin{align*}
    \tilde{\varphi} (d) (\trivtree) & = \psm (\tilde{\varphi} (d)) + \esm (\tilde{\varphi} (d)) \\
    & = \varphi (d) + \esm' (d).
  \end{align*}
  Moreover, for any grading trees $ t $ and $ u $, by the definition of $ \dsm $, we have
  \begin{align*}
    \tilde{\varphi} (d) (t \tjoin u) & = \sum _i (\tilde{\varphi} (d) _{\rmL (i)} (t) \ajoin \tilde{\varphi} (d) _{\rmR (i)} (u)) \\
    & = \sum _i (\tilde{\varphi} (d _{\rmL (i)}) (t) \ajoin \tilde{\varphi} (d _{\rmR (i)}) (u)).
  \end{align*}
  For all elements $ d \in D $ and grading trees $ t $, these two formulae inductively determine $ \tilde{\varphi} (d) (t) $.
  This implies that such $ \tilde{\varphi} $ is unique if it exists.

  For any $ d \in D $, these formulae indeed define a homomorphism $ \tilde{\varphi} (d) \cl \treet \bbK \to \treet (V \oplus \bbK) $, because such a homomorpism can be defined by setting the image of each grading tree.
  Thus it remains to show that $ \tilde{\varphi} (d) $ is representative.
  
  Define a generating tree $ \sigma (d) $ as follows.
  First let $ \sigma (d) \lP \empseq \rP \coloneqq \esm' (d) + \varphi (d) $.
  For each position sequence $ r $, let $ \sigma (d) \lP r \rP \coloneqq \esm' (d _r) + \varphi (d _r) $ if $ d _r $ is defined.
  Then $ \sigma (d) $ generates $ \tilde{\varphi} (d) $, which implies that $ \tilde{\varphi} (d) $ is representative. 
\end{proof}

\subsection{Cofree coalgebras}

In this subsection, we will present cofree coalgebras by constructing the largest admissible subprecoalgebras of cofree coalgebras.

Fix a module $ V $ and a generating tree $ \sigma $.
Take a grading tree $ t $ with $ n \coloneqq \leaf (t) $, and let $ p _1, \ocdots, p _n $ be the leaves of $ t $.
As in the previous subsection, for each leaf $ p _k \ssp (1 \leq k \leq n) $, let $ p _k \eqqcolon \xi _{k, 1} \cdots \xi _{k, m _k} $ denote the representation as a finite sequence of $ \rmL $ or $ \rmR $, and then set
\begin{gather*}
  r _k \coloneqq \xi _{k, 1} (\iota \lP \empseq \rP) \, \xi _{k, 2} (\iota \lP \xi _{k, 1} \rP) \cdots \xi _{k, m _k} (\iota \lP \xi _{k, 1} \cdots \xi _{k, m _k - 1} \rP),
\end{gather*}
to obtain a position sequence, where the variable of the form $ \iota \lP q \rP $ ranges over nonnegative integers.
In addition to it, take $ \alpha _k $ which is either the projection $ \prV \cl V \oplus \bbK \to V $ or $ \prK \cl V \oplus \bbK \to \bbK $.
Then we have choosen the $ n $-tuple $ \vec{\alpha} \coloneqq (\alpha _1, \ocdots, \alpha _n) $, which we call a \textem{projection tuple}.
Let $ a \coloneqq \pdeg (\vec{\alpha}) $ denote the number of $ \prV $'s among them.
Now define
\begin{gather*}
  \sigma \lP t \mP \vec{\alpha} \rP \coloneqq \sum _\iota (\alpha _1 (\sigma \lP r _1 \rP) \otimes \cdots \otimes \alpha _n (\sigma \lP r _n \rP)).
\end{gather*}
This is an element of the tensor product of $ a $ copies of $ V $ and $ n - a $ copies of $ \bbK $.
Since such an element can be canonically identified with an element of $ V ^{\otimes a} $, we regard $ \sigma \lP t \mP \vec{\alpha} \rP $ as an element of $ V ^{\otimes a} $.

\begin{definition}
  Fix a module $ V $.
  A generating tree $ \sigma $ is said to be \textem{weakly normal} when, for any grading trees $ t $ and $ u $,
  and for any projection tuples $ \vec{\alpha} \coloneqq (\alpha _1, \ocdots, \alpha _{\leaf (t)}) $ and $ \vec{\beta} \coloneqq (\beta _1, \ocdots, \beta _{\leaf (u)}) $ which satisfy $ \pdeg (\vec{\alpha}) = \pdeg (\vec{\beta}) $,
  we have $ \sigma \lP t \mP \vec{\alpha} \rP = \sigma \lP u \mP \vec{\beta} \rP $.
  
  Moreover, a generating tree $ \sigma $ is said to be \textem{normal} when, for any position sequence $ s $, the generating tree $ \sigma \untree _s $ is weakly normal.
\end{definition}

\begin{definition}
  For a module $ V $, a homomorphism $ f \cl \treet \bbK \to \treet (V \oplus \bbK) $ is said to be \textem{normal} when $ f $ is generated by a normal generating tree.
  Let $ \nor{V} $ denote the set of all normal homomorphisms.
\end{definition}

Here recall that, for generating trees $ \sigma, \tau $ and a scalar $ \lambda $, we defined $ \sigma + \tau $ and $ \lambda \sigma $ in the proof of \cref{submod}.

\begin{lemma} \label{projsum}
  Fix a module $ V $, and take generating trees $ \sigma $ and $ \tau $.
  For a grading tree $ t $ with $ n \coloneqq \leaf (t) $, a projection tuple $ \vec{\alpha} \coloneqq (\alpha _1, \ocdots, \alpha _n) $ and a position sequence $ s $, we have
  \begin{gather*}
    (\sigma + \tau) \untree _s \lP t \mP \vec{\alpha} \rP =
    \begin{cases}
      \sigma \lP t \mP \vec{\alpha} \rP + \tau \lP t \mP \vec{\alpha} \rP & (s = \empseq) \\
      \sigma \untree _{\xi (i) s'} \lP t \mP \vec{\alpha} \rP & (s = \xi (i) s',\ i \leq m) \\
      \tau \untree _{\xi (i - m) s'} \lP t \mP \vec{\alpha} \rP & (s = \xi (i) s',\ i > m),
    \end{cases}
  \end{gather*}
  where $ m $ is half the number of the children of the root of $ \sigma $.
\end{lemma}

\begin{proof}
  First consider the case $ s = \empseq $.
  Write $ \rho \coloneqq (\sigma + \tau) \untree _s = \sigma + \tau $ for simplicity.
  If $ t = \trivtree $, we have
  \begin{align*}
    \rho \lP t \mP \vec{\alpha} \rP & = \alpha _1 (\rho \lP \empseq \rP) \\
    & = \alpha _1 (\sigma \lP \empseq \rP + \tau \lP \empseq \rP) \\
    & = \sigma \lP t \mP \vec{\alpha} \rP + \tau \lP t \mP \vec{\alpha} \rP.
  \end{align*}
  Otherwise, for each $ 1 \leq k \leq n $, let
  \begin{align*}
    r _k & \coloneqq \xi _{k, 1} (\iota \lP \empseq \rP) \, \xi _{k, 2} (\iota \lP \xi _{k, 1} \rP) \cdots \xi _{k, m _k} (\iota \lP \xi _{k, 1} \cdots \xi _{k, m _k - 1} \rP) \\
    r' _k & \coloneqq \xi _{k, 2} (\iota \lP \xi _{k, 1} \rP) \cdots \xi _{k, m _k} (\iota \lP \xi _{k, 1} \cdots \xi _{k, m _k - 1} \rP),
  \end{align*}
  where $ \xi _{k, l} $ is as above.
  Now we have
  \begin{gather*}
    \rho \lP t \mP \vec{\alpha} \rP \coloneqq \sum _\iota (\alpha _1 (\rho \lP \xi _{1, 1} (\iota \lP \empseq \rP) r' _1 \rP) \otimes \cdots \otimes \alpha _n (\rho \lP \xi _{n, 1} (\iota \lP \empseq \rP) r' _n \rP)).
  \end{gather*}
  By the definition of $ \sigma + \tau $, we have
  \begin{gather*}
    \rho \lP \xi _{k, 1} (\iota \lP \empseq \rP) r' _k \rP =
    \begin{cases}
      \sigma \lP \xi _{k, 1} (\iota \lP \empseq \rP) r' _k \rP & (\iota \lP \empseq \rP \leq m) \\
      \tau \lP \xi _{k, 1} (\iota \lP \empseq \rP) r' _k \rP & (\iota \lP \empseq \rP > m).
    \end{cases}
  \end{gather*}
  Thus the sum above is decomposed as
  \begin{align*}
    \rho \lP t \mP \vec{\alpha} \rP = {} & \sum _\iota (\alpha _1 (\sigma \lP \xi _{1, 1} (\iota \lP \empseq \rP) r' _1 \rP) \otimes \cdots \otimes \alpha _n (\sigma \lP \xi _{n, 1} (\iota \lP \empseq \rP) r' _n \rP)) \\
    & {} + \sum _\iota (\alpha _1 (\tau \lP \xi _{1, 1} (\iota \lP \empseq \rP) r' _1 \rP) \otimes \cdots \otimes \alpha _n (\tau \lP \xi _{n, 1} (\iota \lP \empseq \rP) r' _n \rP)),
  \end{align*}
  which implies $ \rho \lP t \mP \vec{\alpha} \rP = \sigma \lP t \mP \vec{\alpha} \rP + \tau \lP t \mP \vec{\alpha} \rP $.

  The statement in the case $ s \neq \empseq $ is obvious by the definition of $ \sigma + \tau $.
\end{proof}

\begin{lemma} \label{projmul}
  Fix a module $ V $, and take a generating tree $ \sigma $ and a scalar $ \lambda $.
  For a grading tree $ t $ with $ n \coloneqq \leaf (t) $, a projection tuple $ \vec{\alpha} \coloneqq (\alpha _1, \ocdots, \alpha _n) $ and a position sequence $ s $, we have
  \begin{gather*}
    (\lambda \sigma) \untree _s \lP t \mP \vec{\alpha} \rP =
    \begin{cases}
      \lambda \, \sigma \untree _s \lP t \mP \vec{\alpha} \rP & (\text{$ s $ is of the form $ \rmL (i _1) \cdots \rmL (i _a) $, $ a \geq 0 $}) \\
      \sigma \untree _s \lP t \mP \vec{\alpha} \rP & (\text{otherwise}).
    \end{cases}
  \end{gather*}
\end{lemma}

\begin{proof}
  Write $ \rho \coloneqq \lambda \sigma $.
  For each $ 1 \leq k \leq n $, take $ r _k $ as above.
  Now we have
  \begin{align*}
    \rho \untree _s \lP t \mP \vec{\alpha} \rP & = \sum _\iota (\alpha _1 (\rho \untree _s \lP r _1 \rP) \otimes \cdots \otimes \alpha _n (\rho \untree _s \lP r _n \rP)) \\
    & = \sum _\iota (\alpha _1 (\rho \lP s r _1 \rP) \otimes \cdots \otimes \alpha _n (\rho \lP s r _n \rP)).
  \end{align*}
  If $ s $ is of the form $ \rmL (i _1) \cdots \rmL (i _a) $, only $ sr _1 $ among $ sr _1, \ocdots, sr _n $ is of the form $ \rmL (i _1) \cdots \rmL (i _b) $.
  Thus by the definition of $ \lambda \sigma $, we have
  \begin{align*}
    \rho \untree _s \lP t \mP \vec{\alpha} \rP & = \sum _\iota (\alpha _1 (\lambda \sigma \lP s r _1 \rP) \otimes \alpha _2 (\sigma \lP s r _2 \rP) \cdots \otimes \alpha _n (\sigma \lP s r _n \rP)) \\
    & = \lambda \sum _\iota (\alpha _1 (\sigma \untree _s \lP r _1 \rP) \otimes \alpha _2 (\sigma \untree _s \lP r _2 \rP) \otimes \cdots \otimes \alpha _n (\sigma \untree _s \lP r _n \rP)) \\
    & = \lambda \, \sigma \untree _s \lP t \mP \vec{\alpha} \rP.
  \end{align*}
  If $ s $ is not of the form $ \rmL (i _1) \cdots \rmL (i _a) $, none of $ sr _1, \ocdots, sr _n $ is of the form $ \rmL (i _1) \cdots \rmL (i _b) $.
  Thus by the similar calculation, we have $ \rho \untree _s \lP t \mP \vec{\alpha} \rP = \sigma \untree _s \lP t \mP \vec{\alpha} \rP $.
\end{proof}

Now to prove the cofreeness of $ \nor{V} $, we generalise the notion of projection tuple and show one property of weakly normal generating trees.
Consider an $ n $-tuple $ \vec{\alpha} \coloneqq (\alpha _1, \ocdots, \alpha _n) $ such that now each $ \alpha _k $ is either $ \prV $ or $ \prK $ or $\rmid _{V \oplus \bbK} $.
Let $ \pcan (\vec{\alpha}) $ denote the tuple obtained by removing $ \prK $'s from $ (\alpha _1, \ocdots, \alpha _n) $.
For a generating tree $ \sigma $ and a grading tree $ t $, we similarly define
\begin{gather*}
  \sigma \lP t \mP \vec{\alpha} \rP \coloneqq \sum _\iota (\alpha _1 (\sigma \lP r _1 \rP) \otimes \cdots \otimes \alpha _n (\sigma \lP r _n \rP)),
\end{gather*}
where $ r _1, \ocdots, r _n $ is the same as before.
Writing $ a $ and $ a' $ for the number of $ \prV $'s and $ \rmid _{V \oplus \bbK} $'s in $ (\alpha _1, \ocdots, \alpha _n) $ respectively, we regard $ \sigma \lP t \mP \vec{\alpha} \rP $ as an element of the tensor product of $ a $ copies of $ V $ and $ a' $ copies of $ V \oplus \bbK $.
For example, in the case $ n = 4 $ and $ \vec{\alpha} = (\prV, \prK, \rmid _{V \oplus \bbK}, \prV) $, we have $ \pcan (\vec{\alpha}) = (\prV, \rmid _{V \oplus \bbK}, \prV) $ and $ \sigma \lP t \mP \vec{\alpha} \rP \in V \otimes (V \oplus \bbK) \otimes V $.

Note that, in particular when $ \vec{\alpha} = (\rmid, \ocdots, \rmid) $, we have
\begin{gather*}
  \sigma \lP t \mP \vec{\alpha} \rP \coloneqq \sum _\iota (\sigma \lP r _1 \rP \otimes \cdots \otimes \sigma \lP r _n \rP).
\end{gather*}
Thus letting $ f $ be the homomorphism generated by $ \sigma $, we have $ f (t) = \sigma \lP t \mP (\rmid, \ocdots, \rmid) \rP $ as elements of $ (V \oplus \bbK) ^{\otimes n} $.

\begin{lemma} \label{nonproj}
  Fix a module $ V $, and take a weakly normal generating tree $ \sigma $.
  For any grading trees $ t $ and $ u $, and for any generalised projection tuples $ \vec{\alpha} \coloneqq (\alpha _1, \ocdots, \alpha _{\leaf (t)}) $ and $ \vec{\beta} \coloneqq (\beta _1, \ocdots, \beta _{\leaf (u)}) $ which satisfy $ \pcan (\vec{\alpha}) = \pcan(\vec{\beta}) $,
  we have $ \sigma \lP t \mP \vec{\alpha} \rP = \sigma \lP u \mP \vec{\beta} \rP $.
\end{lemma}

\begin{proof}
  Let $ a $ and $ a' $ be the number of $ \prV $'s and $ \rmid _{V \oplus \bbK} $'s in $ (\alpha _1, \ocdots, \alpha _n) $ respectively.
  By binomial expansion, $ \sigma \lP t \mP \vec{\alpha} \rP $ can be regarded as an element of
  \begin{gather*}
    \bigoplus _{0 \leq k \leq a'} \coeff{a'}{k} V ^{\otimes (a + k)},
  \end{gather*}
  where $ \coeff{a'}{k} $ is the binomial coefficient, and $ \coeff{a'}{k} V ^{\otimes (a + k)} $ is the direct sum of $ \coeff{a'}{k} $ copies of $ V ^{\otimes (a + k)} $.
  Then its $ V ^{\otimes (a + k)} $-component is of the form $ \sigma \lP t \mP \vec{\alpha} ^\circ \rP $ for some (not generalised) projection tuple $ \vec{\alpha} ^\circ $ with $ a + k = \deg (\vec{\alpha} ^\circ) $.
  Since $ \sigma $ is weakly normal, such $ \sigma \lP t \mP \vec{\alpha} ^\circ \rP $ depends only on $ a + k $.
  This implies that $ \sigma \lP t \rP $ depends only on $ a $ and $ a' $, which proves the lemma.
\end{proof}

Now we will prove that the set $ \nor{V} $ gives the cofree coalgebra over a module $ V $.
This will be done by three steps.

\begin{lemma} \label{norprec}
  For a module $ V $, the set $ (\nor{V}, \dsm, \esm) $ is a subprecoalgebra of $ (\rep{V}, \dsm, \esm) $.
\end{lemma}

\begin{proof}
  Take normal generating trees $ \sigma, \tau $ and a scalar $ \lambda $.
  To prove that $ \nor{V} $ is closed under summation and scalar multiplication, we need to show that $ \sigma + \tau $ and $ \lambda \sigma $ are also normal.
  It immediately follows by \cref{projsum,projmul}.

  It remains to show that $ \dsm (f) \in \nor{V} \otimes \nor{V} $ holds for any element $ f \in \nor{V} $.
  Suppose $ f $ is generated by $ \sigma $.
  For each $ i $, let $ g _i $ be the homorphism generated by $ \sigma \untree _{\rmL (i)} $ and $ h _i $ by $ \sigma \untree _{\rmR (i)} $.
  By the definition of normality, it follows that $ \sigma \untree _{\rmL (i)} $ and $ \sigma \untree _{\rmR (i)} $ are normal, and thus $ g _i $ and $ h _i $ are also normal.
  Now by the definition of $ \dsm $, we have
  \begin{gather*}
    \dsm (f) = \sum _i (g _i \otimes h _i),
  \end{gather*}
  and thus $ \dsm (f) $ is in $ \nor{V} \otimes \nor{V} $.
\end{proof}

\begin{lemma} \label{noradm}
  For a module $ V $, the precoalgebra $ (\nor{V}, \dsm, \esm) $ is admissible.
\end{lemma}

\begin{proof}
  To prove the coassociativity of $ \dsm $, we need to show that, for any normal homomorphism $ f $, we have
  \begin{gather*}
    \sum _{i, j} (f _{\rmL (i) \rmL (j)} \otimes f _{\rmL (i) \rmR (j)} \otimes f _{\rmR (i)}) = \sum _{i, j} (f _{\rmL (i)} \otimes f _{\rmR (i) \rmL (j)} \otimes f _{\rmR (i) \rmR (j)}).
  \end{gather*}
  By \cref{teninj}, it suffices to show that, for any grading trees $ t, u, v $, we have
  \begin{gather*}
    \sum _{i, j} (f _{\rmL (i) \rmL (j)} (t) \otimes f _{\rmL (i) \rmR (j)} (u) \otimes f _{\rmR (i)} (v)) = \sum _{i, j} (f _{\rmL (i)} (t) \otimes f _{\rmR (i) \rmL (j)} (u) \otimes f _{\rmR (i) \rmR (j)} (v)),
  \end{gather*}
  as elements of $ (V \oplus \bbK) ^{\otimes (\leaf (t) + \leaf (u) + \leaf (v))} $.
  Supposing that $ f $ is generated by a normal generating tree $ \sigma $, this equation is rewritten as
  \begin{gather*}
    \sigma \lP (t \tjoin u) \tjoin v \mP (\rmid, \ocdots, \rmid) \rP = \sigma \lP t \tjoin (u \tjoin v) \mP (\rmid, \ocdots, \rmid) \rP,
  \end{gather*}
  which holds by \cref{nonproj}.

  To prove the counitality of $ \esm $, we need to show that, for any normal homomorphism $ f $, we have
  \begin{gather*}
    f = \sum _i \prK (f _{\rmL (i)} (\trivtree)) f _{\rmR (i)} = \sum _i \prK (f _{\rmR (i)} (\trivtree)) f _{\rmL (i)}.
  \end{gather*}
  To this end, it suffices to show that, for any grading tree $ t $,
  \begin{gather*}
    f (t) = \sum _i \prK (f _{\rmL (i)} (\trivtree)) f _{\rmR (i)} (t) = \sum _i \prK (f _{\rmR (i)} (\trivtree)) f _{\rmL (i)} (t)
  \end{gather*}
  as elements of $ (V \oplus \bbK) ^{\otimes {\leaf (t)}} $.
  Supposing that $ f $ is generated by a normal generating tree $ \sigma $, this equation is rewritten as
  \begin{gather*}
    \sigma \langle t \mP (\rmid, \ocdots, \rmid) \rangle = \sigma \lP \trivtree \tjoin t \mP (\prK, \rmid, \ocdots, \rmid) \rP = \sigma \lP t \tjoin \trivtree \mP (\rmid, \ocdots, \rmid, \prK) \rP,
  \end{gather*}
  which again holds by \cref{nonproj}.
\end{proof}

\begin{theorem} \label{cofmodtwo}
  For a module $ V $, the coalgebra $ (\nor{V}, \dsm, \esm) $ is cofree over $ V $.
\end{theorem}

Before the proof, we recall the notation defined before.
For a subprecoalgebra $ (F, \dsm, \esm) $ of $ (\nor{V}, \dsm, \esm) $ and an element $ f \in F $, we have written
\begin{gather*}
  \delta (f) \eqqcolon \sum _i (f _{\rmL (i)} \otimes f _{\rmR (i)}).
\end{gather*}
Suppose that $ f $ is generated by $ \sigma $.
For each $ i $, let $ g _i $ be the homorphism generated by $ \sigma \untree _{\rmL (i)} $ and $ h _i $ by $ \sigma \untree _{\rmR (i)} $.
Now since we have
\begin{gather*}
  \delta (f) \eqqcolon \sum _i (g _i \otimes h _i),
\end{gather*}
we can choose $ f _{\rmL (i)} \coloneqq g _i $ and $ f _{\rmR (i)} \coloneqq h _i $.
In the next proof, we adopt this choice for $ f _{\rmL (i)} $'s and $ f _{\rmR (i)} $'s.
Then we have $ f _{\rmL (i)} (\trivtree) = \sigma \untree _{\rmL (i)} \lP \empseq \rP = \sigma \lP \rmL (i) \rP $ and similarly $ f _{\rmR (i)} (\trivtree) = \sigma \lP \rmR (i) \rP $.
Recall also that we have defined $ f _r $ for a position sequence $ r $.
Adopting the choice above, we have $ f _r (\trivtree) = \sigma \untree _r \lP \empseq \rP = \sigma \lP r \rP $.

We also introduce a new notation here.
Take an admissible subprecoalgebra $ (F, \dsm, \esm) $ of $ (\nor{V}, \dsm, \esm) $.
Consider a morphism $ F \to F ^{\otimes n} \ssp (n \geq 1) $ which is given as a composite of morphisms of the form $ \rmid ^{\otimes a} \otimes \dsm \otimes \rmid ^{\otimes b} \cl F ^{\otimes (a + b + 1)} \to F ^{\otimes (a + b + 2)} \ssp (a, b \geq 0) $.
Since $ \dsm \cl F \to F \otimes F $ is coassociative, such a morphism is independent of the way it is constructed, and depends only on $ n $.
Thus we write $ \lM \dsm \rM ^n \cl F \to F ^{\otimes n} $ for that morphism.

Now we will prove \cref{cofmodtwo}.

\begin{proof}
  By \cref{barr}, it suffices to show that $ (\nor{V}, \dsm, \esm) $ is the largest admissible subprecoalgebra of $ (\rep{V}, \dsm, \esm) $.
  To this end, for a admissible subprecoalgebra $ (F, \dsm, \esm) $ and an element $ f \in F $, we will prove that $ f $ is normal.
  Let $ \sigma $ be a generating tree which generates $ f $.

  First we will prove that $ f $ is weakly normal.
  Take a grading tree $ t $ with $ n \coloneqq \leaf (t) $ and a projection tuple $ \vec{\alpha} \coloneqq (\alpha _1, \ocdots, \alpha _n) $ with $ a \coloneqq \pdeg (\vec{\alpha}) $.
  We have
  \begin{align*}
    f (t) & = \sum _\iota (\sigma \lP r _1 \rP \otimes \sigma \lP r _2 \rP \otimes \cdots \otimes \sigma \lP r _n \rP) \\
    & = \sum _\iota (f _{r _1} (\trivtree) \otimes f _{r _2} (\trivtree) \otimes \cdots \otimes f _{r _n} (\trivtree)) \\
    & = \theta ^{\otimes n} (\lM \dsm \rM ^n (f)),
  \end{align*}
  where $ \theta \cl F \to V \oplus \bbK $ is the morphism defined by $ \theta (f) \coloneqq f (\trivtree) $.
  Thus it follows that $ \sigma \lP t \mP \vec{\alpha} \rP $ is an image of $ f $ by the morphism obtained by the composite
  \begin{gather*}
    \begin{cd}{}
      F \ar[r, "\lM \dsm \rM ^n"] \& F ^{\otimes n} \ar[r, "\theta ^{\otimes n}"] \& (V \oplus \bbK) ^{\otimes n} \ar[r, "\alpha _1 \otimes \cdots \otimes \alpha _n"] \&[1.5em] V ^{\otimes a} \otimes \bbK ^{\otimes (n - a)} \ar[r] \& V ^{\otimes a}, \tag{$ \sharp _1 $}
    \end{cd}
  \end{gather*}
  where the last morphism is the canonical isomorphism.
  We will show that this morphism is equal to the composite
  \begin{gather*}
    \begin{cd}{}
      F \ar[r, "\lM \dsm \rM ^a"] \& F ^{\otimes a} \ar[r, "\pi ^{\otimes a}"] \& V ^{\otimes a}. \tag{$ \sharp _2 $}
    \end{cd}
  \end{gather*}
  If this is done, it follows that $ \sigma \lP t \mP \vec{\alpha} \rP $ depends only on $ a $, and thus $ \sigma $ is shown to be weakly normal.

  For simplicity, we shall only illustrate the case in which $ \alpha _1, \ocdots, \alpha _a $ are $ \prV $ and $ \alpha _{a + 1}, \ocdots, \alpha _n $ are $ \prK $, but the argument in the other cases is similar.
  Consider the diagram
  \begin{gather*}
    \begin{cd}{row sep=large, column sep=large}
      F \ar[r, "\lM \dsm \rM ^n"] \ar[d, "\lM \dsm \rM ^a"'] \& F ^{\otimes n} \ar[rd, "\theta ^{\otimes n}"] \ar[d, "\rmid ^{\otimes a} \otimes \esm ^{\otimes (n - a)}"'] \& \\
      F ^{\otimes a} \ar[r] \ar[rd, "\psm ^{\otimes a}"'] \& F ^{\otimes a} \otimes \bbK ^{\otimes (n - a)} \ar[dr, pos=0.4, "\psm ^{\otimes a} \otimes \rmid ^{\otimes (n - a)}"'] \& (V \oplus \bbK) ^{\otimes n} \ar[d, "\alpha _1 \otimes \cdots \otimes \alpha _n"] \\
      \& V ^{\otimes a} \ar[r] \& V ^{\otimes a} \otimes \bbK ^{\otimes (n - a)} \makebox[0em]{,}
   \end{cd} 
  \end{gather*}
  where the unlabelled morphisms are the canonical isomorphisms.
  The upper-left square commutes by the coassociativity of $ \dsm $ and the counitality of $ \esm $, and the other two parallelograms also commute clearly.
  Hence the whole diagram commutes, which means that the two morphisms $ \sharp _1 $ and $ \sharp _2 $ are equal.

  Next we will show that $ f $ is normal.
  For any position sequence $ s $, the homomorphism $ f _s $ is that generated by $ \sigma \untree _s $.
  Since $ f _s $ is also in $ F $, the same argument as above proves that $ \sigma \untree _s $ is weakly normal.
  This concludes that $ \sigma $ is normal, and thus $ f $ is normal.
\end{proof}

\section{Comparison with the construction given by Block--Leroux}

In this section, we will compare our construction of cofree coalgebras with the other known construction.
Since cofree coalgebras are unique up to isomorphism, our cofree coalgebras are isomorphic to those given by Sweedler, Block--Leroux, Hazewinkel and Murfet \cite{hax,bgdc,hccm,mscc}.
Among them, we will explicitly present the isomorphism between our cofree coalgebras and those by Block--Leroux.
First let us recall the construction given by them.
Here we suppose that $ \bbK $ is a field.

\begin{definition}
  For a vector space $ V $, a linear map $ f \cl \bbK [X] \to TV $ is said to be \textem{representative} when $ f $ preserves degrees and there exist two finite families $ (g _i) _{i \in I} $ and $ (h _i) _{i \in I} $ of linear maps $ \bbK [X] \to TV $ such that, for any elements $ t, u \in \bbK [X] $, we have
  \begin{gather*}
    f (t \cdot u) = \sum _{i \in I} (g _i (t) \cdot h _i (u)).
  \end{gather*}
  Here $ \bbK [X] $ denotes the polynomial algebra over $ \bbK $.
  Let $ \blrep{V} $ denote the set of all representative linear maps\footnotemark.
  \footnotetext{In the original paper, $ \blrep{V} $ is denoted by $ \bbK [X] ^0 _V $.}
\end{definition}

For a linear map $ f \cl \bbK [X] \to TV $ which is representative in the sence of Block--Leroux, take $ (g _i) _{i \in I} $ and $ (h _i) _{i \in I} $ as in the definition above.
They showed that
\begin{align*}
  \dbl (f) & \coloneqq \sum _{i \in I} (g _i \otimes h _i) \\
  \ebl (f) & \coloneqq f (1) \\
  \pbl (f) & \coloneqq f (X)
\end{align*}
yield three well-defined linear maps
\begin{align*}
  \dbl \cl {} & \blrep{V} \to \blrep{V} \otimes \blrep{V} \\
  \ebl \cl {} & \blrep{V} \to \bbK \\
  \pbl \cl {} & \blrep{V} \to V,
\end{align*}
and these data form the cofree coalgebra over $ V $:

\begin{theorem}[Block--Leroux, Theorem 2 \cite{bgdc}]
  For a vector space $ V $, the data $ (\blrep{V}, \dbl, \ebl, \pbl) $ gives the cofree coalgebra over $ V $.
\end{theorem}

For a vector space $ V $, Block and Leroux's $ \blrep{V} $ and our $ \nor{V} $ are both cofree over $ V $, which implies that $ \blrep{V} $ and $ \nor{V} $ are mutually isomorphic.
This isomorphism can be explicitly illustrated as follows.

If a generating tree $ \sigma $ is normal, then for a grading tree $ t $ and a projection tuple $ \vec{\alpha} $, the value $ \sigma \lP t \mP \vec{\alpha} \rP $ depends only on $ a \coloneqq \pdeg (\vec{\alpha}) $.
Thus we simply write $ \sigma \lP a \rP $ for that value.

Now let $ f $ denote the linear map generated by $ \sigma $.
For each integer $ n \geq 1 $, define
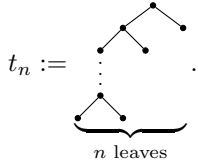
\newcommand{\tntree}{
  \tikz[baseline={([yshift=-.8ex]current bounding box.center)}]{
    \path (0, 0) node [point] (O) {};
    \path (-0.4, -0.3) node [point] (L) {};
    \path (0.4, -0.3) node [point] (R) {};
    \path (-0.7, -0.6) node [point] (LL) {};
    \path (-0.1, -0.6) node [point] (LR) {};
    \path (-0.7, -1.2) node [point] (X) {};
    \path (-1, -1.5) node [point] (XL) {};
    \path (-0.4, -1.5) node [point] (XR) {};
    \draw (O) -- (L); \draw (O) -- (R);
    \draw (L) -- (LL); \draw (L) -- (LR);
    \draw (X) -- (XL); \draw (X) -- (XR);
    \path (-0.7, -0.8) node {\scriptsize$ \vdots $}
  }
}
\begin{gather*}
  t _n \coloneqq \underbrace{\tntree} _{\text{$ n $ leaves}}.
\end{gather*}
Then we have
\begin{gather*}
  \sigma \lP n \rP = \sigma \lP t _n \mP (\prV, \ocdots, \prV) \rP = \prV ^{\otimes n} (f (t _n)),
\end{gather*}
where $ \prV \cl V \oplus \bbK \to V $ is the projection, and $ f (t _n) $ is regarded as an element of $ (V \oplus \bbK) ^{\otimes n} $.
This ensures that the value $ \sigma \lP n \rP $ is independent of the choice of $ \sigma $ and determined only by $ f $.
Moreover, we have
\begin{gather*}
  \sigma \lP 0 \rP = \sigma \lP \trivtree \mP (\prK) \rP = \prK (f (\trivtree)),
\end{gather*}
where $ \prK \cl V \oplus \bbK \to \bbK $ is the projection.
This means that $ \sigma \lP 0 \rP $ is also determined only by $ f $.
Hence the mapping
\begin{gather*}
  \begin{array}{r@{}r@{}c@{}l}
    \tilde{\varphi} \cl {} & \nor{V} & {} \longrightarrow {} & \blrep{V} \\
    & f & {} \longmapsto {} & \lb
    \begin{array}{@{\,}r@{}c@{}l@{\,}}
      \bbK [X] & {} \longrightarrow {} & TV \\
      X ^n & {} \longmapsto {} & \sigma \lP n \rP
    \end{array}
    \rb
  \end{array}
\end{gather*}
is well-defined and linear by \cref{projsum,projmul}.
Since $ \sigma \lP 1 \rP = \prV (f (\trivtree)) $, the diagram
\begin{gather*}
  \begin{cd}{}
    \nor{V} \ar[d, "\tilde{\varphi}"'] \ar[dr, "\psm"] \& \\
    \blrep{V} \ar[r, "\pbl"'] \& V 
  \end{cd}
\end{gather*}
commutes, and thus $ \tilde{\varphi} $ is the desired isomorphism.

The correspondence presented by $ \tilde{\varphi} $ can be also restated as follows.
For each $ f \in \nor{V} $, which is generated by a generating tree $ \sigma $, we can identify $ f $ with the family $ (\sigma \lP n \rP) _{n \in \mathbb{N}} $ in the aforementioned way.
On the other hand, we can also identify each $ g \in \blrep{V} $ with the family $ (g (X ^n)) _{n \in \mathbb{N}} $.
Then the statement that $ f \in \nor{V} $ corresponds with $ g \in \blrep{V} $ exactly means that the families $ (\sigma \lP n \rP) _{n \in \mathbb{N}} $ and $(g (X ^n)) _{n \in \mathbb{N}} $ coincide.

Thus if we regard $ \nor{V} $ and $ \blrep{V} $ as sets of those identified families, Block and Leroux's construction and ours state in a different way the condition which a family needs to satisfy to belong to $ \nor{V} $ or $ \blrep{V} $.

\section*{Acknowledgements}

This work is supported by Leading Graduate Course for Frontiers of Mathematical Sciences and Physics.
I would like to thank my supervisor, Associate Professor Ryu Hasegawa, for his thoughtful guidance.

\bibliographystyle{plain}
\bibliography{reference}

\end{document}